\documentclass{article}
\usepackage{a4wide}
\usepackage{amsmath}
\usepackage{amssymb}
\usepackage{amsthm}
\usepackage{bbm}
\usepackage{enumerate}
\usepackage[colorlinks=true,urlcolor=blue,linkcolor=blue,plainpages=false,pdfpagelabels]{hyperref}

\newcommand{\vv}[1]{\overrightarrow{#1}}

\newtheorem{theorem}{Theorem}[section]
\newtheorem{proposition}[theorem]{Proposition}
\newtheorem{lemma}[theorem]{Lemma}
\newtheorem{corollary}[theorem]{Corollary}

\newtheorem{definition}[theorem]{Definition}

\newcommand{\be}{\begin{enumerate}}
\newcommand{\ee}{\end{enumerate}}
\newcounter{ct}

\newcommand{\beq}{\begin{equation}}
\newcommand{\eeq}{\end{equation}}
\newcommand {\bua} {\begin{eqnarray*}}
\newcommand {\eua} {\end {eqnarray*}}

\newcommand{\details}[1]{}
\newcommand{\limn}{\ds\lim_{n\to\infty}}
\newcommand{\ds}{\displaystyle}
\def\N{{\mathbb N}}

\title{An abstract proximal point algorithm}
\author{Lauren\c tiu Leu\c stean${}^{a,b,c}$, Adriana Nicolae${}^{d,e}$ and Andrei Sipo\c s${}^{f,c}$\\[0.2cm]
\footnotesize ${}^a$The Research Institute of the University of Bucharest (ICUB), University of Bucharest, \\
\footnotesize M. Kog\u{a}lniceanu 36-46, 050107, Bucharest, Romania\\
\footnotesize ${}^b$Faculty of Mathematics and Computer Science, University of Bucharest,\\
\footnotesize Academiei 14, 010014, Bucharest, Romania\\[0.1cm]
\footnotesize ${}^c$Simion Stoilow Institute of Mathematics of the Romanian Academy,\\
\footnotesize Calea Grivi\c{t}ei 21, 010702, Bucharest, Romania\\[0.1cm]
\footnotesize ${}^d$Department of Mathematical Analysis - IMUS, University of Seville,\\
\footnotesize C/ Tarfia s/n, 41012 Sevilla, Spain\\[0.1cm]
\footnotesize ${}^e$Department of Mathematics, Babe\c s-Bolyai University,\\
\footnotesize Kog\u alniceanu 1, 400084 Cluj-Napoca, Romania\\[0.1cm]
\footnotesize ${}^f$Department of Mathematics, Technische Universit\"at Darmstadt,\\
\footnotesize Schlossgartenstrasse 7, 64289 Darmstadt, Germany\\
\footnotesize E-mails: laurentiu.leustean@unibuc.ro,  anicolae@math.ubbcluj.ro, sipos@mathematik.tu-darmstadt.de\\
}
\date{}

\begin{document}

\maketitle

\begin{abstract}
The proximal point algorithm is a widely used tool for solving a variety of convex optimization problems such as finding zeros of maximally monotone operators, fixed points of nonexpansive mappings, as well as minimizing convex functions. The algorithm works by applying successively so-called ``resolvent'' mappings associated to the original object that one aims to optimize. In this paper we abstract from the corresponding resolvents employed in these problems the natural notion of jointly firmly nonexpansive families of mappings. This leads to a streamlined method of proving weak convergence of this class of algorithms in the context of complete CAT(0) spaces (and hence also in Hilbert spaces). In addition, we consider the notion of uniform firm nonexpansivity in order to similarly provide a unified presentation of a case where the algorithm converges strongly. Methods which stem from proof mining, an applied subfield of logic, yield in this situation computable and low-complexity rates of convergence.\\


\end{abstract}

\section{Introduction}

The first instance of what came later to be known as the proximal point algorithm can be found 
in a short communication from 1970 of Martinet \cite{Mar70}. He considered (among others) the 
issue of solving the minimization problem
$${\arg\!\min}_{x \in C} f(x),$$
where $C$ is a closed convex subset of a Hilbert space $H$ and $f$ is a (real-valued) lower 
semicontinuous convex function defined on $C$, that further has the property that for all 
$a \in \mathbb{R}$, the set
$$\{x \in C \mid f(x) \leq a\}$$
is bounded. One then starts from an arbitrary point $x_0 \in C$ and afterwards iteratively 
builds a sequence $(x_n)$ by the implicit (though uniquely determining) recurrence relation
$$f(x_{n+1}) = \min_{y\in C} (f(y) + \|x_n - y\|^2).$$
Martinet's Th\'eor\`eme 3 then asserts that any weak cluster point of this sequence is a 
solution to the given minimization problem.

In 1976, Rockafellar \cite{Roc76} took up the more general problem of finding a zero of a 
maximally monotone multi-valued operator $A: H \to 2^H$, i.e. a point $x$ such that $0 \in A(x)$ 
(single-valued monotone operators had already been considered by Martinet). This contains 
the previous case since the subdifferential $\partial f$ of a function $f$ having the properties 
considered above is a maximally monotone operator whose zeros coincide with the minimizers of $f$. 
The method used in this case in order to approach the desired solution was called the 
``proximal point algorithm'' and  generates starting from a point $x_0 \in C$ a sequence using another implicit recurrence, namely
$$x_n \in (id_H + \gamma_nA)(x_{n+1}),$$
where $(\gamma_n)$ is a sequence of positive real numbers. When $A = \partial f$, the relation 
reduces to the previous one if $(\gamma_n)$ is the sequence constantly equal to $1/2$. 
Theorem 1 of \cite{Roc76} shows that if
$$\inf_{n \in \mathbb{N}} \gamma_n > 0,$$
then $(x_n)$ weakly converges to a zero of $A$. Strong convergence is proved under some 
additional uniformity assumptions (such as $A^{-1}$ being Lipschitz continuous at $0$), but 
it does not generally hold, as G\"uler \cite[Corollary 5.1]{Gul90} later put forward a counterexample 
in this sense. Two years after Rockafellar's paper, Br\'ezis and Lions \cite{BreLio78} studied 
more general conditions one could impose on $(\gamma_n)$ that still yield weak convergence 
of the sequence $(x_n)$, there regarded as the ``infinite product'' of the resolvent operators
$$J_{\gamma_nA}:=(id_H + \gamma_nA)^{-1}.$$
Those conditions continue to be the state of the art -- e.g. for an arbitrary maximally monotone 
operator one may assume (\cite[Proposition 8]{BreLio78})
$$\sum_{n=0}^{\infty} \gamma_n^2 = \infty.$$

The proximal point algorithm has grown to become a versatile tool of convex optimization, 
being used, in addition to the applications already expounded upon, to solve a plethora of 
problems such as variational inequalities, minimax or equilibrium problems (some of these 
may already be found in the papers cited above). The book of Bauschke and Combettes \cite{BauCom10} 
may serve as an introduction to the field in the context of Hilbert spaces.

Outside Hilbert spaces, a natural generalization of the resolvent operator was given in the 
1990s by Jost \cite{Jos95} and Mayer \cite{May98} in the context of complete CAT(0) spaces, 
which can be regarded as the proper nonlinear analogue of Hilbert spaces. Using this definition 
and an appropriate notion of weak convergence introduced by Lim \cite{Lim76}, called 
$\Delta$-convergence, Ba\v{c}\'ak \cite{Bac13} extended in 2013 Th\'eor\`eme 9 of Br\'ezis and 
Lions \cite{BreLio78} in this context. More precisely, he proved the $\Delta$-convergence of 
the sequence generated by the proximal point algorithm when $f$ is a proper, convex and lower 
semicontinuous function that attains its minimum.

A separate strand of development came from fixed point theory. In the 1960s, Browder \cite{Bro67} 
and Halpern \cite{Hal67} studied the existence and computation of fixed points of nonexpansive 
mappings $T: C \to C$ (where $C$ is a closed convex bounded subset of a Hilbert space). They 
considered the notion of a resolvent of order $\gamma$ of $T$ -- that is, a mapping satisfying, 
for all $x \in C$,
$$R_{\gamma}x = \frac1{1+\gamma}x + \frac{\gamma}{1+\gamma}TR_{\gamma}x.$$
In the above, we reparametrized their construction in order to obtain a better fit with the 
objects considered here. Their main result states that by letting $\gamma \to \infty$, $R_{\gamma}x$ 
tends to the fixed point of $T$ which is the closest to $x$. Halpern's particularly simple 
argument was later generalized to the Hilbert ball in \cite{GoeRei84} and to complete CAT(0) 
spaces in \cite{Kir03}. A proof in the latter setting that starts from minimal boundedness 
assumptions may be found in a 2014 paper of Ba\v{c}\'ak and Reich \cite{BacRei14}. Note that 
\cite{BacRei14} also contains a variant of the proximal point algorithm which constructs, by 
iterating the resolvents of $T$, a sequence that $\Delta$-converges to a fixed point of $T$.

As one may notice, every iterative sequence that was considered above under the name of 
``proximal point algorithm'' follows a pattern: we have a mathematical object that we seek 
to optimize in some way, we construct associated ``resolvent'' operators, we take an initial 
arbitrary point $x$ and finally we iterate those operators starting from $x$. One may ask 
whether there are some very general hypotheses which yield the convergence of the resulting 
sequence without explicitly considering the particular details of the optimization problem 
at hand. Our first main goal is to answer this question in the affirmative in the framework 
of CAT(0) spaces (and therefore also for Hilbert spaces) by deriving some conditions related 
to firm nonexpansivity that are satisfied by all the above types of families of resolvents. 
We give these conditions in weaker and stronger forms, and show that while the strongest one 
is generally satisfied, the weakest one suffices to obtain appropriate convergence results. 
We should mention here that it was known for a long time that individual resolvents are in 
particular firmly nonexpansive, and some abstract results in the same spirit were previously 
obtained by Ariza-Ruiz, the first author and L\'opez-Acedo in \cite{AriLeuLop14}. The present 
paper may be regarded as a natural continuation of the study initiated there (see also \cite{AriLopNic15}).

Section~\ref{sec:prelim} introduces general notions and properties regarding geodesic metric 
spaces and mappings that are used in the sequel. Section~\ref{sec:gen-PPA} starts with some 
very general hypotheses for a sequence generated by a family of firmly nonexpansive mappings 
that yield its weak or $\Delta$-convergence. In the process, we derive some lemmas that characterize 
various asymptotic aspects of the proximal point algorithm. Then we define two conditions 
that one may impose on a family $(T_n)$ with respect to a sequence $(\gamma_n)$. These conditions 
generalize the property of a mapping being firmly nonexpansive to a relation between two possibly 
different mappings which is then applied to each possible pair from the countable family. We claim 
that these definitions capture the residual property used in convergence proofs that corresponds to 
the way a family of resolvents $(J_{\gamma_n})$ behaves with respect to the sequence of step-sizes 
$(\gamma_n)$. We consider then ``jointly firmly nonexpansive families'' and a somewhat weaker notion, 
``jointly $(P_2)$ families'' from which the general conditions can be obtained. In particular, 
the families of mappings involved in the problems discussed before (i.e. minimization of convex 
functions, finding fixed points of nonexpansive mappings and finding zeros of maximally monotone operators) 
satisfy these conditions.

The second main goal of this paper is to find quantitative variants of some convergence results 
for the proximal point algorithm. This falls within the purview of proof mining, an applied subfield 
of logic. Proof mining primarily concerns itself with the application of tools from proof theory 
to obtain computational content for theorems in ordinary mathematics with proofs that are not 
necessarily fully constructive. The project was first suggested in the 1950s by Kreisel under 
the name of ``unwinding of proofs'', but it gained considerable momentum after its extensive development 
in the 1990s and 2000s by Kohlenbach and his collaborators, culminating with the publication of 
general logical metatheorems, developed by Kohlenbach \cite{Koh05} and by Gerhardy and Kohlenbach 
\cite{GerKoh08}, that tell us when a proof of a theorem proven in classical logic may be analyzed 
in order to obtain (``extract'') its hidden quantitative information. A comprehensive reference for 
the major developments of the field up to 2008 is the monograph of Kohlenbach \cite{Koh08}, while 
surveys of recent results are \cite{Koh17,Koh18}. So far, proof mining has been successfully applied to 
obtain quantitative versions of celebrated results in various areas of mathematics such as approximation 
theory, nonlinear analysis,  metric fixed point theory, ergodic theory, or topological dynamics. 
Recently, its methods have begun to be applied to 
convex optimization, for more details see \cite{BacKohXX, KohLeuNic18, KohLopNic17, KohLopNicXX, Kou17, LeuSip18,LeuSip18b}.

Let us discuss the sort of quantitative results that we obtain.  If $(x_n)$ is a sequence in 
a metric space $X$ and $x\in X$, then $\lim_{n\to\infty}x_n=x$ 
if and only if
$$\forall k \in {\mathbb N} \,\exists N \in {\mathbb N} \,\forall n \geq N\, \left(d(x_n,x) 
\leq \frac1{k+1}\right).$$
A quantitative version of the above would be a {\it rate of convergence} for the sequence: a 
formula showing how to compute the $N$ in terms of the $k$. However, very simple real-valued 
sequences have been shown by methods of mathematical logic to lack a computable rate of convergence. 
We recall, though, from the discussion above, that strong convergence of the proximal point algorithm 
could only be obtained under some extra uniformity assumptions. Fortunately, some of these conditions 
yield the uniqueness of the needed optimizing point (minimizer, fixed point or zero). This uniqueness 
was shown by the work of Kohlenbach \cite{Koh90}, Kohlenbach and Oliva \cite[Section 4.1]{KohOli03} and Briseid \cite{Bri09}
to guarantee the extraction of a rate of convergence, relative to some other piece of quantitative information.
In Section~\ref{sec:uniform}, therefore, we define a general notion of uniformity applicable to our families of mappings
(extending the similar notion given in \cite{BarBauMofWan16} in the context of Hilbert spaces). One then shows that
the concrete algorithms  have corresponding ``uniform'' cases that are subsumed into this definition, e.g. finding
zeros of uniformly monotone mappings or minimizing uniformly convex functions.

Section~\ref{sec:quantitative} then shows that this definition suffices: a quantitative variant 
of the asymptotic lemmas from Section~\ref{sec:gen-PPA} fits in as the relevant piece of information 
that is then used, as per the above discussion, to obtain a highly uniform rate of convergence for 
this special case of the proximal point algorithm. As a byproduct, we obtain an alternate proof 
for the classical qualitative results of strong convergence.

Proximal methods are not limited to the classical problems of convex optimization. Therefore, a question that arises is to what extent a natural and abstract approach of the type provided here could be employed to capture other such variants, which are used, for example, in global (non-convex) optimization \cite{KapTic98,IusPenSva03}, where the necessity of the existence of iterates requires one to assume weak forms of monotonicity. Another direction consists in considering multi-valued resolvent-type operators instead of single-valued ones. In this case the algorithm becomes nondeterministic (i.e. given a current iterate, the following one is not uniquely determined). Such a development would allow one to cover e.g. vector-valued optimization problems \cite{BonIusSva05,CenMorYao10}.

\section{Preliminaries}\label{sec:prelim}

We start by briefly recalling some notions and properties about geodesic spaces needed in the sequel. 
More details on geodesic spaces can be found, for example,  in \cite{Pap05,BriHae99,Bac14}. 
Let $(X,d)$ be a metric space. A {\em geodesic} in $X$ is a mapping $\gamma : [a,b] \to X$ (where $a,b \in \mathbb{R}$)
 such that for all $s,t \in [a,b]$,
$$d(\gamma(s),\gamma(t))=|s-t|.$$ 
We say that $X$ is a {\em geodesic space} if for all $x,y \in X$, there is a geodesic 
$\gamma : [a,b] \to X$ satisfying $\gamma(a)=x$ and $\gamma(b)=y$.

A geodesic space $(X,d)$ is called a {\em CAT(0) space} if for all $z \in X$, all geodesics 
$\gamma : [a,b] \to X$ and all $t \in [0,1]$ we have that
\beq
d^2(z,\gamma((1-t)a+tb)) \leq (1-t)d^2(z,\gamma(a)) + td^2(z,\gamma(b)) - t(1-t)d^2(\gamma(a),\gamma(b)). \label{def-CAT0}
\eeq
It easily follows that every CAT(0) space is {\it uniquely geodesic} -- that is,  for any $x,y$ in 
such a space $X$ there is a {\em unique} geodesic $\gamma : [0,d(x,y)] \to X$ such that 
$\gamma(0)=x$ and $\gamma(d(x,y))=y$ -- and in this framework we shall denote, for any $t \in [0,1]$, 
the point $\gamma(td(x,y))$ by $(1-t)x + ty$. Note that every CAT(0) space $X$ is {\it Busemann convex} -- i.e., for any $x,y,u,v \in X$ and $t \in [0,1]$, 
\begin{equation}\label{mcv}
d((1-t)x + ty,(1-t)u + tv) \leq (1-t)d(x,u) + td(y,v).
\end{equation}

We will also make use of the {\it quasi-linearization function} $\langle\cdot,\cdot\rangle : X^2 \times X^2 \to \mathbb{R}$ introduced by Berg and Nikolaev in \cite{BerNik08}, which is defined, for any $x,y,u,v \in X$, by the following (where an ordered pair of points $(w,w') \in X^2$ is denoted by $\vv{ww'}$):
\begin{equation}\label{def-quasilin-fct}
\langle \vv{xy}, \vv{uv} \rangle := \frac12(d^2(x,v) + d^2(y,u) - d^2(x,u) -d^2(y,v)).
\end{equation}
\begin{proposition}[{\cite[Proposition 14]{BerNik08}}]\label{bnchar}
In any metric space $(X,d)$, the mapping $\langle\cdot,\cdot\rangle$ is the unique one that satisfies, for any $x,y,u,v,w \in X$, the following properties:
\begin{enumerate}[(i)]
\item $\langle\vv{xy},\vv{xy}\rangle = d^2(x,y)$;
\item $\langle\vv{xy},\vv{uv}\rangle = \langle\vv{uv},\vv{xy}\rangle$;
\item $\langle\vv{yx},\vv{uv}\rangle = -\langle\vv{xy},\vv{uv}\rangle$;
\item $\langle\vv{xy},\vv{uv}\rangle + \langle\vv{xy},\vv{vw}\rangle = \langle\vv{xy},\vv{uw}\rangle$.
\end{enumerate}
\end{proposition}
In particular, if $X$ is a real Hilbert space with inner product $\langle\cdot,\cdot\rangle$, then
\begin{equation}\label{eq-quasi-inner}
\langle\vv{xy},\vv{uv}\rangle = \langle x-y,u-v \rangle = \langle y-x,v-u \rangle,
\end{equation} 
for all $x,y,u,v \in X$. This justifies the notation.

The main result of \cite{BerNik08}, Theorem 1, gives a characterization of CAT$(0)$ spaces in terms of the ``Cauchy-Schwarz'' inequality for
 $\langle\cdot,\cdot\rangle$. More precisely, a geodesic space $(X,d)$ is CAT$(0)$ if and only if  
\begin{equation}\label{CauchySchwartz}
\langle\vv{xy},\vv{uv}\rangle \leq d(x,y)d(u,v),
\end{equation}
for all $x,y,u,v \in X$. Furthermore, by \cite[Theorem 6]{BerNik08}, a related condition for a geodesic space $(X,d)$ to be CAT$(0)$ is the following inequality
\begin{equation}\label{bn}
d^2(x,v) + d^2(y,u) \leq d^2(x,u) + d^2(y,v) + d^2(x,y) + d^2(u,v),
\end{equation}
which is to be satisfied for all $x,y,u,v \in X$.

\mbox{} 

For the rest of the section, $(X,d)$ is a geodesic space, unless stated otherwise. If $T:X\to X$ is a mapping, we denote by $Fix(T)$ the set of its fixed points.

The following generalization of firmly nonexpansive mappings to geodesic spaces was introduced in \cite{AriLeuLop14}.

\begin{definition}\label{def-fn}
A mapping $T : X \to X$ is called {\em firmly nonexpansive} if for any $x,y \in X$ and any $t \in [0,1]$ we have that
$$d(Tx,Ty) \leq d((1-t)x + tTx,(1-t)y+tTy).$$
\end{definition}

As mentioned in \cite{AriLopNic15} (see also \cite{KohLopNicXX}), if $X$ is a CAT$(0)$ space, every firmly nonexpansive mapping $T:X \to X$ satisfies the so-called {\em property $(P_2)$}. Namely, 
$$2d^2(Tx,Ty) \leq d^2(x,Ty) + d^2(y,Tx) - d^2(x,Tx) - d^2(y,Ty),$$
for all $x,y \in X$. In other words,
\begin{equation}\label{eq-quasi-P2}
d^2(Tx,Ty) \leq \langle\vv{TxTy},\vv{xy}\rangle,
\end{equation}
for all $x,y \in X$. If $X$ is a Hilbert space, property $(P_2)$ is sufficient for firm nonexpansivity 
as \eqref{eq-quasi-P2} and \eqref{eq-quasi-inner} yield $\|Tx - Ty\|^2 \le \langle Tx-Ty,x-y \rangle$, 
which is, in turn, equivalent to Definition \ref{def-fn} (see, e.g., \cite[Proposition 4.2]{BauCom10} 
for a proof). Moreover, from \eqref{eq-quasi-P2} and \eqref{CauchySchwartz} one immediately obtains the following result. 
\begin{lemma}\label{P2-im-nonexp}
If $X$ is a CAT(0) space and $T:X\to X$ satisfies property $(P_2)$, then $T$ is nonexpansive.
\end{lemma}

Let $(x_n)$ be a bounded sequence in $X$ and $F\subseteq X$ be nonempty. For any $y\in X$, define
\[r(y, (x_n)):= \limsup_{n \to \infty} d(y,x_n), \quad r(F, (x_n)) := \inf \{ r(y, (x_n)) \mid y \in F\}.\]
Furthermore, $A(F,(x_n)) := \{y \in F \mid r(y, (x_n)) = r(F,(x_n))\}$ and elements of $A(F,(x_n))$
are called {\em asymptotic centers} of  $(x_n)$  with respect to $F$.
We shall denote $A(X,(x_n))$ by $A((x_n))$ and call its elements asymptotic centers of $(x_n)$.

The next results will be used in the subsequent sections. 

\begin{lemma}[{\cite[Lemma 3.2]{Leu10}}]\label{UCW-useful-unique-as-center}
Let $(x_n)$ be a bounded sequence in $X$ with $A((x_n))=\{c\}$ and  $(\alpha_n),(\beta_n)$ be 
real sequences such that $\alpha_n\geq 0$ for all $n\in\N$,
$\limsup_{n \to \infty} \alpha_n\leq 1$ and $\limsup_{n \to \infty} \beta_n\leq 0$.\\
 Assume that $y\in X$  is such that there exist $p,N\in\N$ satisfying, for all $n\geq N$, 
\[d(y,x_{n+p})\leq \alpha_nd(c,x_n)+\beta_n.\]
Then $y=c$.
\end{lemma}

\begin{proposition}[{\cite[Proposition 7]{DhoKirSim06}}]\label{CAT0-unique-ac}
Every bounded sequence $(x_n)$ in  a complete CAT(0) space $X$ has a unique asymptotic center 
with respect to any nonempty closed convex subset of $X$.
\end{proposition}

In order to state our main results, we need to introduce the notion of $\Delta$-convergence 
which was defined by Lim \cite{Lim76} in metric spaces. We refer to \cite{Kuc80,Jos94,EspFer09} 
for equivalent notions in the setting of complete CAT(0) spaces, where $\Delta$-convergence can 
be seen as a generalization of the weak convergence in Banach spaces (see \cite{KirPan08}). 
In fact, in Hilbert spaces, $\Delta$-convergence coincides with weak convergence 
(see \cite[Exercise 3.1]{Bac14}). 

\begin{definition}
A bounded sequence $(x_n)$ {\em $\Delta$-converges} to a point $x \in X$ if for any subsequence  
$(u_n)$ of  $(x_n)$ we have that $A((u_n))=\{x\}$.
\end{definition}

The notion of Fej\'{e}r monotonicity will also play an important role in this work. Let 
$(x_n)$ be a sequence in $X$ and $F\subseteq X$ be nonempty.

\begin{definition}
We say that $(x_n)$ is {\it Fej\'er monotone} with respect to $F$ if for all $p \in F$ and 
all $n \in \mathbb{N}$, we have that
$$d(x_{n+1},p)\leq d(x_n,p).$$
\end{definition}
It is obvious that if $(x_n)$ is Fej\'er monotone with respect to $F$, then $(d(x_n,p))$ 
converges for every $p\in F$ and, furthermore,  $(x_n)$ is bounded.

Finally, we recall the following  well-known result (see, for example, \cite[Proposition 3.3.(iii)]{BacSeaSim12}), 
which turns out to be very useful in obtaining $\Delta$-convergence results.

\begin{proposition} \label{as-cen-unic-Delta}
Let $X$ be a complete CAT(0) space and $(x_n)$ be Fej\'er monotone with respect to $F$. Assume that 
the asymptotic center of every subsequence of $(x_n)$ is in $F$.
Then $(x_n)$ $\Delta$-converges to some $x\in F$.
\end{proposition}

\section{An abstract Proximal Point Algorithm}\label{sec:gen-PPA}

We now begin the process of modularizing the proof(s) that guarantee the weak convergence of 
common instances of the proximal point algorithm. Theorem~\ref{gppa} is the first stage in 
this sense and provides some highly general conditions under which the iteration constructed 
by applying countably many mappings converges weakly. In proving it, we shall also show 
some fundamental properties of that iterative sequence, such as Fej\'er monotonicity and a 
form of asymptotic regularity.

\mbox{} 

In the following, $X$ is  a complete CAT(0) space and $(T_n)_{n \in \mathbb{N}}$ is a family of self-mappings of $X$ 
satisfying property $(P_2)$ and having common fixed points. Set 
$$F:=\bigcap_{n \in \mathbb{N}} Fix(T_n)\neq\emptyset.$$

For $x\in X$, we define the following iteration starting with $x$:
\beq
x_0:=x, \quad x_{n+1}:=T_nx_n \text{~for all~}n\in\mathbb{N}. \label{def-main-iteration}
\eeq
Let  $(\gamma_n)$ be a sequence of positive real numbers such that $\sum_{n=0}^\infty \gamma_n^2 = \infty$. 

The following conditions will also be considered in the sequel:
\begin{enumerate}
\item[$(C1)$] for all $n,m \in \mathbb{N}$ and  $w\in X$, $d(T_nw,T_mw) \leq \frac{|\gamma_n-\gamma_m|}{\gamma_n}d(w,T_nw)$; 
\item[$(C2)$]  the sequence $\left(\frac{d(x_n,x_{n+1})}{\gamma_n}\right)_{n \in \mathbb{N}}$ is nonincreasing.
\end{enumerate}

We include below a series of preliminary results. 

\begin{lemma}\label{fixtn}
Suppose that $(C1)$ holds. Then $Fix(T_n) = F$ for every $n \in \N$.
\end{lemma}
\begin{proof}
It follows immediately.  
\details{Obviously, $F\subseteq Fix(T_n)$ for every $n$. 
If $p$ is a fixed point of a $T_n$ and $m \in \mathbb{N}$, then
$d(p,T_mp)=d(T_np,T_mp)\leq \frac{|\gamma_n-\gamma_m|}{\gamma_n}d(p,p) = 0$. Thus, $p$ is a fixed point of every $T_m$.}
\end{proof}

\begin{lemma}\label{2-lem-1}
For all $p\in F$ and all $n \in \mathbb{N}$, we have that
$$d^2(x_{n+1},p) \leq d^2(x_n,p) - d^2(x_n,x_{n+1}).$$
In particular, $(x_n)$ is Fej\'er monotone with respect to $F$.
\end{lemma}
\begin{proof}
Let $p\in F$ and $n \in \mathbb{N}$. Since $T_n$ satisfies property $(P_2)$ and $p\in Fix(T_n)$, 
we have that $2d^2(T_nx_n,p) \leq d^2(x_n,p) + d^2(T_nx_n,p) - d^2(x_n,T_nx_n)$. It follows that  
$d^2(T_nx_n,p) \leq d^2(x_n,p) - d^2(x_n,T_nx_n)$, hence the conclusion.
\end{proof}

\begin{lemma}\label{2-lem-2}
Assume that $(C2)$ is satisfied. Then 

$$\limn d(x_n,x_{n+1})=0 \quad \text{and}\quad \lim_{n \to \infty} \frac{d(x_n,x_{n+1})}{\gamma_n} = 0.$$
\end{lemma}
\begin{proof}
Since $F\ne\emptyset$, there exists $p\in F$. Let $b>0$ be such that $d(x,p)\leq b$. For every $n\in \mathbb{N}$, we have, by Lemma~\ref{2-lem-1}, that 
\bua 
\sum_{k=0}^n d^2(x_k,x_{k+1}) \leq \sum_{k=0}^n ( d^2(x_k,p) - d^2(x_{k+1},p)) =  d^2(x,p) - d^2(x_{n+1},p) 
\leq b^2.
\eua
It follows that the series $\sum_{n=0}^\infty d^2(x_n,x_{n+1})$ converges, so $\limn d(x_n,x_{n+1})=0$.

We  prove now that $\lim_{n \to \infty} \frac{d(x_n,x_{n+1})}{\gamma_n}=0$. Let $\varepsilon > 0$. 
Since $\sum_{n=0}^\infty \gamma_n^2 = \infty$, there exists $N\in\N$ such that
$\sum_{k=0}^N \gamma_k^2 \geq b^2/\varepsilon^2$. If for all $k\in\{0,\ldots, N\}$ 
one has that $\frac{d(x_k,x_{k+1})}{\gamma_k}>\varepsilon$, we get that 
$$\sum_{k=0}^Nd^2(x_k,x_{k+1}) > \sum_{k=0}^{N}\gamma_k^2 \varepsilon^2 \geq b^2,$$
a contradiction.  Hence, there exists $M\in\{0,\ldots, N\}$ such that 
$\frac{d(x_M,x_{M+1})}{\gamma_M} \leq \varepsilon$. 
Since, by $(C2)$, the sequence $\left(\frac{d(x_n,x_{n+1})}{\gamma_n}\right)$ is nonincreasing, we get that 
$\frac{d(x_n,x_{n+1})}{\gamma_n} \leq \varepsilon$ for all $n \geq M$.
\end{proof}

\begin{proposition}\label{2-asreg}
Assume that $(C1)$ and $(C2)$ are satisfied. Then for all $m \in \mathbb{N}$, 
$$\lim_{n \to \infty} d(x_n,T_mx_n) = 0.$$
\end{proposition}
\begin{proof}
Let $m \in \mathbb{N}$. We get that for all $n \in \mathbb{N}$,
\bua
d(x_n,T_mx_n) &\leq & d(x_n,x_{n+1}) + d(x_{n+1},T_mx_n) = d(x_n,x_{n+1}) + d(T_nx_n, T_mx_n) \\
&\leq & d(x_n,x_{n+1}) + \frac{|\gamma_n-\gamma_m|}{\gamma_n}d(x_n,T_nx_n) \quad \text{by }(C1)\\
&\leq & 2d(x_n,x_{n+1}) + \gamma_m \cdot \frac{d(x_n,x_{n+1})}{\gamma_n}.
\eua
Our conclusion follows by applying Lemma~\ref{2-lem-2}.
\end{proof}

We can prove now the main result of this section.

\begin{theorem}[Abstract Proximal Point Algorithm]\label{gppa}
Let $X$ be a complete CAT(0) space and $(T_n)$ be a family of self-mappings of $X$ 
satisfying property $(P_2)$ and having common fixed points. Set $F:=\bigcap_{n \in \mathbb{N}} Fix(T_n)\neq\emptyset$.
For $x\in X$, let $(x_n)$ be defined by \eqref{def-main-iteration}.
Let  $(\gamma_n)$ be a sequence of positive real numbers such that $\sum_{n=0}^\infty \gamma_n^2 = \infty$. 
Assume that $(C1)$ and $(C2)$ hold.

Then  $(x_n)$ $\Delta$-converges to a point in $F$.
\end{theorem}
\begin{proof}
Note first that by  Lemma~\ref{2-lem-1}, $(x_n)$ is Fej\'er monotone with respect to $F$, hence bounded. Let $(u_n)$ be an arbitrary subsequence of $(x_n)$. By Proposition~\ref{CAT0-unique-ac}, $(u_n)$ has a unique asymptotic center $u$. 
We shall prove that $u\in F$, so let $m\in\N$ be arbitrary. Note that
\[d(T_mu,u_n)\leq d(T_mu,T_mu_n)+d(u_n,T_mu_n)\leq d(u,u_n)+d(u_n,T_mu_n).\]
Applying Lemma~\ref{UCW-useful-unique-as-center} with $\alpha_n=1$, $\beta_n=d(u_n,T_mu_n)$, $p=N=0$ and using the fact that
$\lim_{n \to \infty} d(u_n,T_mu_n)=0$ (by Proposition~\ref{2-asreg}), we get that $T_mu=u$.

Finally, Proposition~\ref{as-cen-unic-Delta} yields that $(x_n)$ $\Delta$-converges to a point in $F$.
\end{proof}

\subsection{Jointly firmly nonexpansive families of mappings}\label{jfne}

We shall now proceed to the second stage of our abstraction -- that is, giving a natural condition for a family $(T_n)$ and a sequence $(\gamma_n)$ of positive numbers such that the previous general conditions are satisfied. This can be regarded as an extension of the project initiated in \cite{AriLeuLop14} with the definition and the asymptotic behaviour of a firmly nonexpansive mapping to the case of a countable family of mappings. Recall that the notion of a firmly nonexpansive mapping in a Hilbert space has two analogues when considered within the more general setting of CAT(0) spaces. In the same spirit, we shall present here two definitions that apply to families of mappings which coincide when restricted to Hilbert spaces.

In the sequel, $X$ is a CAT(0) space, $T_n: X \to X$ for every $n\in \N$ and $(\gamma_n)$ is a sequence of 
positive real numbers.

\begin{definition}
The family $(T_n)$ is said to be {\em jointly firmly nonexpansive with respect to} $(\gamma_n)$ 
if for all $n, m \in \mathbb{N}$, $x,y \in X$ and all $\alpha, \beta \in [0,1]$ such that 
$(1-\alpha)\gamma_n=(1-\beta)\gamma_m$, 
\beq
d(T_nx,T_my) \leq d((1-\alpha)x+\alpha T_nx,(1-\beta)y+\beta T_my). \label{def-jfne}
\eeq
\end{definition}

\begin{definition}
We say that the family $(T_n)$ is {\em jointly $(P_2)$ with respect to} $(\gamma_n)$ if for all 
$n, m \in \mathbb{N}$ and all $x,y \in X$, 
\beq
 \frac1{\gamma_m}(d^2(T_nx,T_my) + 
d^2(y,T_my) - d^2(y,T_nx)) \leq \frac1{\gamma_n}(d^2(x,T_my) - d^2(x,T_nx) - d^2(T_nx,T_my)).\label{def-jP2}
\eeq
\end{definition}

\begin{lemma}\label{jointly-implies-single}
If $(T_n)$ is jointly firmly nonexpansive (resp. jointly $(P_2)$) with respect to 
$(\gamma_n)$, then each $T_n$ is firmly nonexpansive (resp. satisfies property $(P_2)$). 
\end{lemma}
\begin{proof}
Apply \eqref{def-jfne} (resp. \eqref{def-jP2}) for $m=n$. In the first case, remark that given $t \in [0,1]$, we take $\alpha = \beta = t$. 
\details{For $P_2$: We get that 
$d^2(T_nx,T_ny) + d^2(y,T_ny) - d^2(y,T_nx)\leq d^2(x,T_ny) - d^2(x,T_nx) - d^2(T_nx,T_ny)$, hence
$2d^2(T_nx,T_ny)\leq d^2(x,T_ny)+d^2(y,T_nx)-d^2(y,T_ny)-d^2(x,T_nx)$. Thus $T_n$ satisfies $P_2$.}
\end{proof}

\begin{proposition}\label{jfn-jp2}
If $(T_n)$ is jointly firmly nonexpansive with respect to $(\gamma_n)$, then $(T_n)$ is 
jointly $(P_2)$ with respect to $(\gamma_n)$.
\end{proposition}
\begin{proof}
Let $m,n \in \mathbb{N}$ and $x,y \in X$. We choose arbitrarily 
$\alpha \in \left(1- \min\left\{\gamma_m/\gamma_n,1\right\}, 1\right)$ 
and set $$\beta:= 1 - (1-\alpha)\frac{\gamma_n}{\gamma_m}.$$
Then $\beta\in (0,1)$ and $(1-\alpha)\gamma_n = (1-\beta)\gamma_m$.
\details{$\beta>0$ iff $(1-\alpha)\frac{\gamma_n}{\gamma_m}<1$ iff 
$1-\alpha<\frac{\gamma_m}{\gamma_n}$ iff 
$\alpha>1-\frac{\gamma_m}{\gamma_n}$, which is true since $\alpha>1- 
\min\left\{\frac{\gamma_m}{\gamma_n},1\right\}\geq 1-\frac{\gamma_m}{\gamma_n}$.
$\beta<1$ iff $(1-\alpha)\frac{\gamma_n}{\gamma_m}>0$ iff $\alpha<1$
Apply now \eqref{def-jfne}.} 
Hence, applying the fact that $(T_n)$ is jointly firmly nonexpansive and the inequality \eqref{def-CAT0} 
twice, we get that 
\begin{align*}
d^2(T_nx,T_my) &\leq  d^2((1-\alpha)x+\alpha T_nx,(1-\beta)y+\beta T_my)\\
&\leq  (1-\alpha)d^2(x,(1-\beta)y+\beta T_my)+\alpha d^2(T_nx,(1-\beta)y+\beta T_my)-\alpha(1-\alpha)d^2(x,T_nx)\\
&\leq  (1-\alpha)(1-\beta)d^2(x,y)+(1-\alpha)\beta d^2(x,T_my)-(1-\alpha)\beta(1-\beta)d^2(y,T_my)+\\
&\ \ \ + \alpha(1-\beta)d^2(T_nx,y)+\alpha\beta d^2(T_nx,T_my)-\alpha\beta(1-\beta)d^2(y,T_my)-\alpha(1-\alpha)d^2(x,T_nx)\\
&=  (1-\alpha)(1-\beta)d^2(x,y) + (1-\beta)\alpha d^2(T_nx,y) + (1-\alpha)\beta d^2(x,T_my) +\\
&\ \ \ + \alpha\beta d^2(T_nx,T_my) - \alpha(1-\alpha)d^2(x,T_nx) - \beta(1-\beta) d^2(y,T_my), 
\end{align*}
so
\begin{align*}
(1-\alpha\beta) d^2(T_nx,T_my) &\leq  (1-\alpha)(1-\beta)d^2(x,y) + (1-\beta)\alpha d^2(T_nx,y) + (1-\alpha)\beta d^2(x,T_my)- \\
&\ \ \  - \alpha(1-\alpha)d^2(x,T_nx) - \beta(1-\beta) d^2(y,T_my).
\end{align*}
Dividing now the above inequality by $1-\alpha > 0$, we obtain that
\begin{align*}
\frac{1-\alpha\beta}{1-\alpha}d^2(T_nx,T_my)& \leq (1-\beta)d^2(x,y) + \frac{(1-\beta)\alpha}{1-\alpha} d^2(T_nx,y) +\beta d^2(x,T_my) -\\
& \ \ \  - \alpha d^2(x,T_nx) - \frac{\beta(1-\beta)}{1-\alpha} d^2(y,T_my).
\end{align*}
By easy computations, one can see that 
$$\frac{1-\alpha\beta}{1-\alpha} = 1+ \alpha\frac{\gamma_n}{\gamma_m}, \quad 
\frac{(1-\beta)\alpha}{1-\alpha} = \alpha\frac{\gamma_n}{\gamma_m} \quad \text{and} \quad
\frac{\beta(1-\beta)}{1-\alpha} = \left(1 - (1 -\alpha)\frac{\gamma_n}{\gamma_m}\right)\frac{\gamma_n}{\gamma_m}.$$
\details{
\bua 
\frac{1-\alpha\beta}{1-\alpha}&=& \frac{1-\alpha\left(1 - (1-\alpha)\frac{\gamma_n}{\gamma_m}\right)}{1-\alpha}=
 \frac{1-\alpha+\alpha(1-\alpha)\frac{\gamma_n}{\gamma_m}}{1-\alpha}=1+\alpha\frac{\gamma_n}{\gamma_m}\\
 \frac{(1-\beta)\alpha}{1-\alpha} &=&  \frac{\left(1-\left(1 - (1-\alpha)\frac{\gamma_n}{\gamma_m}\right)\right)\alpha}
 {1-\alpha}
=\alpha\frac{\gamma_n}{\gamma_m}\\
\frac{\beta(1-\beta)}{1-\alpha} &=& 
\frac{\left(1 - (1 -\alpha)\frac{\gamma_n}{\gamma_m}\right)(1 -\alpha)\frac{\gamma_n}{\gamma_m}}{1-\alpha}
= \left(1 - (1 -\alpha)\frac{\gamma_n}{\gamma_m}\right)\frac{\gamma_n}{\gamma_m}
\eua}
Therefore, we have that 
\begin{align*}
\left(1+ \alpha\frac{\gamma_n}{\gamma_m}\right)d^2(T_nx,T_my) &\leq (1 -\alpha)\frac{\gamma_n}{\gamma_m}d^2(x,y) + 
\alpha\frac{\gamma_n}{\gamma_m}d^2(T_nx,y) + \left(1 - (1 -\alpha)\frac{\gamma_n}{\gamma_m}\right) d^2(x,T_my) -\\
& \ \ \  - \alpha d^2(x,T_nx) - \left(1 - (1 -\alpha)\frac{\gamma_n}{\gamma_m}\right)\frac{\gamma_n}{\gamma_m} d^2(y,T_my).
\end{align*}
Letting $\alpha \to 1$, we get that
$$
\left(1+ \frac{\gamma_n}{\gamma_m}\right)d^2(T_nx,T_my) \leq   \frac{\gamma_n}{\gamma_m} d^2(T_nx,y) +d^2(x,T_my) 
- d^2(x,T_nx) - \frac{\gamma_n}{\gamma_m} d^2(y,T_my),
$$
so
\begin{align*}
\frac{\gamma_n}{\gamma_m}\left(d^2(T_nx,T_my)+d^2(y,T_my)-d^2(T_nx,y)\right) \leq &  d^2(x,T_my)-d^2(x,T_nx)-d^2(T_nx,T_my).
\end{align*}
Divide by $\gamma_n$ to obtain \eqref{def-jP2}, our required inequality.
\end{proof}

Using the quasi-linearization function defined by \eqref{def-quasilin-fct}, the joint $(P_2)$ condition can equivalently be expressed as:
\begin{equation}\label{jp2}
\frac1{\gamma_m}\langle\vv{T_nxT_my},\vv{yT_my}\rangle\leq \frac1{\gamma_n}\langle\vv{T_nxT_my},\vv{xT_nx}\rangle,
\end{equation}
for all $n,m \in \N$.

\begin{proposition}\label{tntp}
Suppose that $(T_n)$ is jointly $(P_2)$ with respect to $(\gamma_n)$. Then for all 
$m,n\in\mathbb{N}$ and all $w \in X$, 
$$d(T_nw,T_mw) \leq \frac{|\gamma_n-\gamma_m|}{\gamma_n}d(w,T_nw).$$
\end{proposition}
\begin{proof}
Let $m,n \in \mathbb{N}$. We shall denote, for simplicity, $T:=T_n$, $U:=T_m$, 
$\lambda:=\gamma_n$, $\mu:=\gamma_m$.

We want to show that for all $w\in X$,
$$d(Tw,Uw) \leq \frac{|\lambda-\mu|}\lambda d(w,Tw).$$
If $Tw=Uw$, the statement is trivially true. Let $w \in X$ be such that $Tw \neq Uw$.\\

\noindent {\bf Claim:} $(\lambda+\mu)d^2(Tw,Uw) \leq (\lambda - \mu)(d^2(w,Tw) - d^2(w,Uw)).$\\[1mm]
{\bf Proof of claim:} 
We have that
$$\frac1\mu\langle \vv{TwUw},\vv{wUw} \rangle  \leq \frac1\lambda\langle \vv{TwUw},\vv{wTw} \rangle ,$$
and, by multiplying with $(-\lambda)$, we get that
\begin{equation}
\langle \vv{TwUw},\vv{Tww} \rangle \leq \frac\lambda\mu \langle \vv{TwUw},\vv{Uww} \rangle.\label{e1}
\end{equation}
A simple expansion of $\langle\cdot,\cdot\rangle$ shows that
\begin{equation}
d^2(Tw,Uw) = d^2(w,Uw) - d^2(w,Tw) + 2\langle \vv{TwUw},\vv{Tww} \rangle.\label{e2}
\end{equation}
By exchanging the roles of $T$ and $U$ in the above equation, we obtain that
\begin{equation}
d^2(Uw,Tw) = d^2(w,Tw) - d^2(w,Uw) + 2\langle \vv{UwTw},\vv{Uww} \rangle.\label{e3}
\end{equation}
Applying \eqref{e1} and \eqref{e2} and multiplying \eqref{e3} by $\frac\lambda\mu$, we get that
\begin{align*} 
d^2(Tw,Uw) &\leq  d^2(w,Uw) - d^2(w,Tw) + \frac{2\lambda}\mu\langle \vv{TwUw},\vv{Uww} \rangle,\\
\frac\lambda\mu d^2(Uw,Tw) &= \frac\lambda\mu d^2(w,Tw) - \frac\lambda\mu d^2(w,Uw) + 
\frac{2\lambda}\mu\langle \vv{UwTw},\vv{Uww} \rangle.
\end{align*}

As a consequence, it follows that
$$\left(1+\frac\lambda\mu\right)d^2(Tw,Uw) \leq \left(\frac\lambda\mu -1\right)(d^2(w,Tw) 
- d^2(w,Uw)).$$
Multiply by $\mu$ to get the claim.\hfill $\blacksquare$\\[2mm]
We distinguish now two cases, according to the sign of $\lambda-\mu$.

When $\lambda - \mu$ is negative, we obtain, using the claim, that
\begin{align*}
(\lambda+\mu)d^2(Tw,Uw) &\leq (\lambda - \mu)(d^2(w,Tw) - d^2(w,Uw))\\
&= (\mu-\lambda)(d^2(w,Uw) - d^2(w,Tw))\\
&\leq (\mu-\lambda)((d(w,Tw)+d(Tw,Uw))^2 - d^2(w,Tw))\\
&=(\mu-\lambda)d(Tw,Uw)(d(Tw,Uw)+2d(w,Tw)).
\end{align*}
Dividing by $d(Tw,Uw)\neq 0$, we have that 
$$(\lambda+\mu)d(Tw,Uw)\leq 2(\mu-\lambda)d(w,Tw) + (\mu-\lambda)d(Tw,Uw),$$
so 
$$2\lambda d(Tw,Uw)\leq 2(\mu-\lambda)d(w,Tw).$$
Thus, 
$$d(Tw,Uw) \leq \frac{\mu-\lambda}\lambda d(w,Tw) = \frac{|\lambda-\mu|}\lambda d(w,Tw),$$
as required.

Now, when $\lambda - \mu$ is positive, we proceed as follows. By the reverse triangle 
inequality for metric spaces, we have that
$$d^2(w,Uw) \geq  | d(Tw,Uw) - d(w,Tw)|^2 = d^2(Tw,Uw) - 2d(w,Tw)d(Tw,Uw) + d^2(w,Tw).$$
Applying the claim, we obtain that
\begin{align*}
(\lambda+\mu)d^2(Tw,Uw) &\leq (\lambda - \mu)(d^2(w,Tw) - d^2(w,Uw))\\
&\leq (\lambda - \mu)(2d(w,Tw)d(Tw,Uw) - d^2(Tw,Uw))\\
&= (\lambda-\mu)d(Tw,Uw)(2d(w,Tw) - d(Tw,Uw)).
\end{align*}
As above, one gets that $$d(Tw,Uw) \leq \frac{\lambda-\mu}\lambda d(w,Tw) = \frac{|\lambda-\mu|}\lambda d(w,Tw).$$
\details{Again, dividing by $d(Tw,Uw)\neq 0$, we get that:
$$(\lambda+\mu)d(Tw,Uw) \leq 2(\lambda-\mu)d(w,Tw) - (\lambda-\mu)d(Tw,Uw),$$
so that:
$$2\lambda d(Tw,Uw)\leq 2 (\lambda-\mu)d(w,Tw)$$
and
$$d(Tw,Uw) \leq \frac{\lambda-\mu}\lambda d(w,Tw) = \frac{|\lambda-\mu|}\lambda d(w,Tw).$$}
\end{proof}

\begin{corollary}
Suppose that $(T_n)$ is jointly $(P_2)$ with respect to $(\gamma_n)$. Then any two mappings of the family have the same set of fixed points.
\end{corollary}

\begin{proof}
It follows from Proposition~\ref{tntp} and Lemma~\ref{fixtn}.
\end{proof}

\begin{proposition}\label{cs} 
Assume that $(T_n)$ is jointly $(P_2)$ with respect to  $(\gamma_n)$. Let $x \in X$ and $(x_n)$ 
be given by \eqref{def-main-iteration}.
Then the sequence $\left(\frac{d(x_n,x_{n+1})}{\gamma_n}\right)$ is nonincreasing.
\end{proposition}
\begin{proof}
Let $n \in \mathbb{N}$. By \eqref{jp2},
$$\frac1{\gamma_{n+1}}\langle \vv{T_nx_nT_{n+1}x_{n+1}},\vv{x_{n+1}T_{n+1}x_{n+1}}\rangle \leq 
\frac1{\gamma_n}\langle \vv{T_nx_nT_{n+1}x_{n+1}},\vv{x_nT_nx_n}\rangle,$$
that is
$$\frac1{\gamma_{n+1}}\langle \vv{x_{n+1}x_{n+2}},\vv{x_{n+1}x_{n+2}}\rangle \leq 
\frac1{\gamma_n}\langle \vv{x_{n+1}x_{n+2}},\vv{x_nx_{n+1}}\rangle.$$
Thus,
\begin{align*}
0 &\leq \frac1{\gamma_n}\langle \vv{x_{n+1} x_{n+2}}, \vv{x_n  x_{n+1}} \rangle - 
\frac{d^2(x_{n+1},x_{n+2})}{\gamma_{n+1}}  \\
&= \gamma_{n+1} \left( \frac1{\gamma_n\gamma_{n+1}}\langle \vv{x_{n+1} x_{n+2}}, \vv{x_n  x_{n+1}} \rangle - \frac{d^2(x_{n+1},x_{n+2})}{\gamma_{n+1}^2} \right) \\
&\leq \gamma_{n+1} \left( \frac{d(x_n,x_{n+1})}{\gamma_n} \cdot \frac{d(x_{n+1},x_{n+2})}{\gamma_{n+1}} - \frac{d^2(x_{n+1},x_{n+2})}{\gamma_{n+1}^2} \right) \quad \text{by }\eqref{CauchySchwartz}\\
&= \gamma_{n+1} \cdot \frac{d(x_{n+1},x_{n+2})}{\gamma_{n+1}} \left( \frac{d(x_n,x_{n+1})}{\gamma_n}  - \frac{d(x_{n+1},x_{n+2})}{\gamma_{n+1}} \right).
\end{align*}
It follows that $\frac{d(x_{n+1},x_{n+2})}{\gamma_{n+1}} \leq \frac{d(x_n,x_{n+1})}{\gamma_n}$.
\details{
If $\frac{d(x_{n+1},x_{n+2})}{\gamma_{n+1}} \neq 0$, $\frac{d(x_n,x_{n+1})}{\gamma_n} - 
\frac{d(x_{n+1},x_{n+2})}{\gamma_{n+1}} \geq 0$, so $\frac{d(x_{n+1},x_{n+2})}{\gamma_{n+1}} 
\leq \frac{d(x_n,x_{n+1})}{\gamma_n}$. Otherwise, if $\frac{d(x_{n+1},x_{n+2})}{\gamma_{n+1}} = 0$, 
clearly $\frac{d(x_{n+1},x_{n+2})}{\gamma_{n+1}} \leq \frac{d(x_n,x_{n+1})}{\gamma_n}$.}
\end{proof}

We give now another abstract version of the Proximal Point Algorithm.

\begin{theorem}\label{appa}
Let $X$ be a complete CAT(0) space, $T_n: X \to X$ for every $n\in \N$ and $(\gamma_n)$ be  a sequence 
of positive real numbers satisfying $\sum_{n=0}^\infty \gamma_n^2 = \infty$.
Assume that the family  $(T_n)$ is jointly $(P_2)$ with respect to $(\gamma_n)$ (in particular, 
$(T_n)$ may be jointly firmly nonexpansive) and that $F:=\bigcap_{n \in \mathbb{N}} Fix(T_n)\neq\emptyset$.
Let $x \in X$ and $(x_n)$  be given by \eqref{def-main-iteration}.

Then $(x_n)$ $\Delta$-converges to a point in $F$.
\end{theorem}
\begin{proof}
By Lemma~\ref{jointly-implies-single}, each $T_n$ satisfies property $(P_2)$. We can now apply  
Theorem \ref{gppa}, as conditions $(C1)$ and $(C2)$ follow from Propositions~\ref{tntp} and \ref{cs}, respectively.
\end{proof}

\subsubsection{The case of Hilbert spaces}

Assume now that $H$ is a Hilbert space with inner product $\langle\cdot,\cdot\rangle$. We show 
next that joint firm nonexpansivity coincides with the joint $(P_2)$ condition.

\begin{proposition}\label{H-jfne=jP2}
Let $(T_n )$ be a family of self-mappings of $H$ and $(\gamma_n)$ be a sequence of positive real numbers. Then 
$(T_n)$ is jointly $(P_2)$ with respect to $(\gamma_n)$ if and only if $(T_n)$ is jointly 
firmly nonexpansive with respect to  $(\gamma_n)$.
\end{proposition}
\begin{proof}
``$\Leftarrow$'' By Proposition~\ref{jfn-jp2}.

``$\Rightarrow$'' Let $m,n\in\mathbb{N}$, $x,y\in H$ and $\alpha,\beta \in [0,1]$ be such that 
$(1-\alpha)\gamma_n = (1-\beta)\gamma_m =: \delta$. 
A simple computation yields the following two identities
$$(1-\alpha)x + \alpha T_n x = T_n x + \frac{\delta}{\gamma_n}(x-T_nx)\quad \text{and}\quad 
(1-\beta)y + \beta T_m y = T_m y + \frac{\delta}{\gamma_m}(y-T_m y).$$
\details{
$T_n x + \frac{\delta}{\gamma_n}(x-T_nx)=T_n x+(1-\alpha)(x-T_nx)=(1-\alpha)x+\alpha T_nx$
}
It follows that 
\begin{align*}
\|((1-\alpha)x + \alpha T_n x) - ((1-\beta)y + \beta T_my)\|^2 =&\ \left\|(T_nx - T_my) + 
\left(\frac\delta{\gamma_n}(x-T_nx) -\frac\delta{\gamma_m}(y-T_m y)\right) \right\|^2 \\
=&\ \|T_nx - T_my\|^2 + \delta^2\left\|\frac1{\gamma_n}(x-T_nx) -\frac1{\gamma_m}(y-T_m y)\right\|^2 \\
&+2\delta\left\langle T_nx - T_my, \frac1{\gamma_n}(x-T_nx) -\frac1{\gamma_m}(y-T_m y) \right\rangle .
\end{align*}
In order to show that the right-hand side is greater than or equal to $\|T_nx - T_my\|^2$, which is what 
we are aiming to prove here, it is sufficient to show that
$$D:=\left\langle T_nx - T_my, \frac1{\gamma_n}(x-T_nx) -\frac1{\gamma_m}(y-T_m y) \right\rangle\geq 0.$$
Remark that 
\begin{align*}
D &= \frac1{\gamma_n}\left\langle T_nx - T_my,x-T_nx\right\rangle  - \frac1{\gamma_m}\left\langle T_nx - T_my,y-T_m y\right\rangle\\
&= \frac1{\gamma_n}\langle\vv{T_nxT_my},\vv{xT_nx}\rangle - \frac1{\gamma_m}\langle\vv{T_nxT_my},\vv{yT_my}\rangle \quad 
\text{by \eqref{eq-quasi-inner}}\\
&\geq 0 \quad \text{by \eqref{jp2}}.
\end{align*}
Thus, $(T_n)$ is jointly firmly nonexpansive with respect to $(\gamma_n)$.
\end{proof}

Since $\Delta$-convergence coincides with weak convergence in Hilbert spaces, we get, as 
an immediate consequence of Theorem~\ref{appa}, the following abstract version of the Proximal Point Algorithm.
\begin{theorem}\label{appa-Hilbert}
Let $H$ be a Hilbert space, $T_n: H \to H$ for every $n\in \N$ and $(\gamma_n)$ be  a sequence 
of positive real numbers satisfying $\sum_{n=0}^\infty \gamma_n^2 = \infty$.
Assume that the family  $(T_n)$  is jointly firmly nonexpansive with respect to $(\gamma_n)$
 and that $F:=\bigcap_{n \in \mathbb{N}} Fix(T_n)\neq\emptyset$.
Let $x \in H$ and $(x_n)$  be given by \eqref{def-main-iteration}.

Then $(x_n)$ converges weakly to a point in $F$.
\end{theorem}

We are now in a position to prove that specific instances of the proximal point algorithm
satisfy the stronger requirement that their associated families of resolvents are jointly 
firmly nonexpansive, thus justifying our choice of definitions. Three concrete problems -- minimizing 
convex functions, finding fixed points of nonexpansive mappings and finding zeros of maximally 
monotone operators -- are used to illustrate this fact. We may then apply Theorems~\ref{appa} 
and \ref{appa-Hilbert} in order to obtain classical weak convergence results for these iterations.

\subsection{Minimizers of convex proper lsc functions}\label{minim-conv-lsc}

In the sequel, $X$ is a complete CAT(0) space and $f: X \to (-\infty, \infty]$ is a convex, proper, lower semicontinuous (lsc) 
function. A point $x\in X$ is said to be a {\em minimizer} of $f$ if $f(x)=\inf_{y\in X} f(y)$. 
The set of minimizers of $f$ is denoted by $Argmin(f)$. 

For any $\gamma > 0$, let us denote, following \cite{Bac13}, 
\[J_\gamma:X \to X, \quad J_\gamma(x) := {\arg\!\min}_{y \in X} \left[f(y) + \frac1{2\gamma} d^2(x,y)\right].\]
The mapping $J_\gamma$, defined in the context of CAT(0) spaces by Jost \cite{Jos95}, is called the {\it (Moreau-Yosida) resolvent} or 
the {\it proximal mapping} of $f$ of order $\gamma$. 

We recall in the following proposition some well-known properties proven in \cite{Jos95}.

\begin{proposition}\label{Jgamaf-prop}
Let  $\gamma >0 $. Then
\be
\item\label{Fix-Jgamma_f=argmin-f} $Fix(J_\gamma) =Argmin(f)$.
\item\label{Jgamaf-ne} $J_\gamma$ is nonexpansive.
\item\label{Jgamaf-jos} For all $x \in X$ and all $t \in [0,1]$, the following holds:
$$J_{(1-t)\gamma}((1-t)x+ tJ_{\gamma}(x))=J_{\gamma}(x).$$
\ee
\end{proposition}

\begin{proposition}\label{Jgammaf-jfne}
Let $(\gamma_n)$ be a sequence of positive real numbers. Then the family $(J_{\gamma_n})$ is 
jointly firmly nonexpansive with respect to $(\gamma_n)$.
\end{proposition}
\begin{proof}
Let $m,n\in\mathbb{N}$, $x,y\in X$ and $\alpha,\beta \in [0,1]$ be such that
$(1-\alpha)\gamma_n = (1-\beta)\gamma_m =:\delta.$
Applying Proposition~\ref{Jgamaf-prop}, we get that 
\begin{align*}
d(J_{\gamma_n}x, J_{\gamma_m} y) &= d(J_{(1-\alpha)\gamma_n}((1-\alpha)x+\alpha J_{\gamma_n } x), 
J_{(1-\beta)\gamma_m}((1-\beta)y+\beta J_{\gamma_m} y)) \\
&= d(J_{\delta}((1-\alpha)x+\alpha J_{\gamma_n} x), J_{\delta}((1-\beta)y+\beta J_{\gamma_m} y)) \\
&\leq d((1-\alpha)x+\alpha J_{\gamma_n} x,(1-\beta)y+\beta J_{\gamma_m} y).
\end{align*}
\end{proof}

As a consequence of Theorem~\ref{appa}, we get the following  $\Delta$-convergence result.

\begin{theorem}\label{ppa-lsc}
Assume that $Argmin(f)\ne\emptyset$ and let $(\gamma_n)$ be a sequence of positive real numbers such that $\sum_{n=0}^\infty \gamma_n^2 = \infty$. 
For any $x \in X$, define the sequence $(x_n)$, starting with $x$, by 
\begin{equation}\label{def-lsc-ppa}
x_0:=x, \quad  x_{n+1}:=J_{\gamma_n}x_n \, \text{~for all }n \in \mathbb{N}.
\end{equation}
Then $(x_n)$ $\Delta$-converges to a minimizer of $f$.
\end{theorem}
\begin{proof}
For all $n\in\N$, put $T_n:=J_{\gamma_n}$. By Proposition~\ref{Jgamaf-prop}.\eqref{Fix-Jgamma_f=argmin-f}, $Fix(T_n) = Argmin(f)$ for all $n \in \N$.
Furthermore, by Proposition~\ref{Jgammaf-jfne}, the family $(T_n)$ is  jointly firmly nonexpansive 
with respect to $(\gamma_n)$. Hence, we may apply Theorem~\ref{appa} to derive our conclusion.
\end{proof}

The above theorem is a slightly weaker variant (with a completely different proof) of a result due to 
Ba\v{c}\'ak \cite[Theorem 1.4]{Bac13}, since one uses here the stronger assumption $\sum_{n=0}^\infty \gamma_n^2 = \infty$ 
instead of $\sum_{n=0}^\infty \gamma_n = \infty$. We point out that an analysis of Ba\v{c}\'ak's 
original statement from the point of view of proof mining was previously carried out in \cite{LeuSip18,LeuSip18b}.

\subsection{Fixed points of nonexpansive mappings}

We proceed now to give another application. Let $X$ be a complete CAT(0) space  and $T : X \to X$ 
be a nonexpansive mapping. 

For $x \in X$ and $\gamma >0$ we define
$$G_{T,x,\gamma} : X \to X, \quad G_{T,x,\gamma}(y) := \frac1{1+\gamma}x+\frac\gamma{1+\gamma}Ty.$$
It is easy to see that this mapping is Lipschitz with constant $\frac\gamma{1+\gamma} \in (0,1)$. 
Therefore it admits a unique fixed point, which we shall denote by $R_{T,\gamma}x$. We have thus
defined a mapping $R_{T,\gamma} : X\to X$, called the {\it resolvent of order $\gamma$} of $T$, 
which satisfies, for any $x \in X$,
\begin{equation}\label{def-RTgamma}
R_{T,\gamma}x = \frac1{1+\gamma}x + \frac\gamma{1+\gamma}TR_{T,\gamma}x.
\end{equation}
We immediately obtain that $Fix(R_{T,\gamma})=Fix(T)$ for all $\gamma>0$.

\begin{proposition}\label{jfn-ne}
Let $(\gamma_n)$ be a sequence of positive real numbers. Then the family $(R_{T,\gamma_n})$ is 
jointly firmly nonexpansive with respect to $(\gamma_n)$.
\end{proposition}
\begin{proof}
Let $m,n\in\mathbb{N}$, $x,y\in X$ and $\alpha,\beta \in [0,1]$ be such that
$(1-\alpha)\gamma_n = (1-\beta)\gamma_m =: \delta$.
Denote 
\[ u:=(1-\alpha)x+\alpha R_{T,\gamma_n}x, \quad v:=(1-\beta)y+\beta R_{T,\gamma_m}y.\]
Then we have to show that 
\begin{equation}\label{ci2-todo}
d(R_{T,\gamma_n}x,R_{T,\gamma_m}y) \leq d(u,v).
\end{equation}
Using \eqref{def-RTgamma} and the definition of $u$, we may apply \cite[Lemma 2.4.(iii)]{AriLeuLop14} to obtain that
$$R_{T,\gamma_n}x = (1-\nu)u + \nu TR_{T,\gamma_n}x,$$
where
$$\nu := \frac{(1-\alpha)\frac{\gamma_n}{1+\gamma_n}}{1-\alpha\cdot\frac{\gamma_n}{1+\gamma_n}} = \frac{\delta}{1+\delta}.$$
We remark that $\nu \neq 1$. Also note that, while the statement of \cite[Lemma 2.4.(iii)]{AriLeuLop14} requires the four points to be pairwise distinct, 
its conclusion is trivial to show in the case of some of them are equal. \details{
We have that 
$$u=(1-\alpha)x+\alpha R_{T,\gamma_n}x$$
and
$$R_{T,\gamma_n}x = \frac1{1+\gamma_n}x + \frac{\gamma_n}{1+\gamma_n}TR_{T,\gamma_n}x$$
Applying \cite[Lemma 2.4.(iii)]{AriLeuLop14} with 
$y:=u, \lambda:=\alpha, x:=x, z:=R_{T,\gamma_n}x; \alpha:=\frac{\gamma_n}{1+\gamma_n}, w:=TR_{T,\gamma_n}x$,
we get that $R_{T,\gamma_n}x=(1-\nu)u+\nu TR_{T,\gamma_n}x$, where 
\bua 
\nu&=&\frac{(1-\alpha)\frac{\gamma_n}{1+\gamma_n}}{1-\alpha\frac{\gamma_n}{1+\gamma_n}}=
\frac{\delta}{1+\gamma_n-\alpha\gamma_n}=\frac{\delta}{1+\delta}.
\eua
}
We  show  similarly that
$$R_{T,\gamma_m}y = (1-\nu)v + \nu TR_{T,\gamma_m}y.$$
Applying \eqref{mcv} and the nonexpansivity of $T$, we get that 
\begin{align*}
d(R_{T,\gamma_n}x,R_{T,\gamma_m}y) &= d((1-\nu)u + \nu TR_{T,\gamma_n}x,(1-\nu)v + \nu TR_{T,\gamma_m}y) \\
&\leq (1-\nu)d(u,v) + \nu d(TR_{T,\gamma_n}x,TR_{T,\gamma_m}y) \\
&\leq (1-\nu)d(u,v) + \nu d(R_{T,\gamma_n}x,R_{T,\gamma_m}y).
\end{align*}
It follows immediately that \eqref{ci2-todo} holds.
\end{proof}

As an immediate application of Theorem~\ref{appa}, we get 

\begin{theorem}\label{ppa-ne}
Assume that  $Fix(T)\ne\emptyset$ and let $(\gamma_n)$ be a sequence of positive real numbers such that 
$\sum_{n=0}^\infty \gamma_n^2 = \infty$.
For any $x \in X$, define the sequence $(x_n)$ by 
\begin{equation*}
x_0:=x, \quad  x_{n+1}:=R_{T,\gamma_n}x_n \, \text{~for all }n \in \mathbb{N}.
\end{equation*}
Then $(x_n)$ $\Delta$-converges to a fixed point of $T$.
\end{theorem}
\details{For all $n$, put $T_n:=R_{T,\gamma_n}$. We have that $Fix(T_n)=Fix(T)$ for all $n$. Thus,
$F=Fix(T)\ne\emptyset$. By Proposition~\ref{jfn-ne}, 
the family $(T_n)$ is jointly firmly nonexpansive with respect to $(\gamma_n)$, 
we may apply Theorem~\ref{appa} to derive our conclusion.
}

We have therefore obtained a new proof of \cite[Proposition 1.5]{BacRei14}.

\subsection{Zeros of maximally monotone operators}\label{concrete-zeros-maximal}

In the following, $H$ is a Hilbert space with inner product $\langle\cdot,\cdot\rangle$ and $A : H \to 2^H$
is a maximally monotone operator. We denote by $zer(A)$ the set of zeros of $A$. Given $\gamma>0$, the resolvent $J_{\gamma A}$ of 
order $\gamma$ of $A$  is defined by 
$$J_{\gamma A}= (id_H + \gamma A)^{-1}.$$
It is well-known (see, e.g., \cite{BauCom10}) that, for every $\gamma > 0$, 
$J_{\gamma A}:H\to H$ is a single-valued 
firmly nonexpansive mapping satisfying $Fix(J_{\gamma A})=zer(A)$.

\begin{proposition}\label{pj-maxmon}
Let $(\gamma_n)$ be a sequence of positive real numbers. Then the family $(J_{\gamma_n A})$ is 
jointly firmly nonexpansive with respect to $(\gamma_n)$.
\end{proposition}
\begin{proof}
By Proposition~\ref{H-jfne=jP2}, we can prove, equivalently, that the family $(J_{\gamma_n A})$
is jointly $(P_2)$ with respect to $(\gamma_n)$. Let $n,m\in\mathbb{N}$ and $x,y\in H$. It is easy to see that 
$$\frac1{\gamma_n}(x-J_{\gamma_n A}x) \in A(J_{\gamma_n A} x)\quad 
\text{and}
\quad \frac1{\gamma_m}(y-J_{\gamma_m A}y) \in A(J_{\gamma_m A} y).$$
\details{Let $x \in H$. Since $J_{\gamma A}x = (id + \gamma A)^{-1}x$, we have that 
$x \in (id+\gamma A)(J_{\gamma A}x)$. We obtain, successively, that $x \in J_{\gamma A}x + \gamma A(J_{\gamma A}x)$, 
that $x - J_{\gamma A}x \in \gamma A(J_{\gamma A}x)$ and that $\frac1\gamma(x-J_{\gamma A}x) \in A(J_{\gamma A} x)$.
}
By the monotonicity of $A$ we obtain that
$$\left\langle J_{\gamma_n A} x - J_{\gamma_m A} y, \frac1{\gamma_n}(x-J_{\gamma_n A}x) - 
\frac1{\gamma_m}(y-J_{\gamma_m A}y) \right\rangle \geq 0,$$
therefore
$$\frac1{\gamma_m} \langle J_{\gamma_n A} x - J_{\gamma_m A} y,
y - J_{\gamma_m A} y \rangle\leq \frac1{\gamma_n}\langle J_{\gamma_n A} x - J_{\gamma_m A} y, x - 
J_{\gamma_n A} x \rangle. $$
\end{proof}

As a consequence of Theorem~\ref{appa-Hilbert} we derive the following well-known weak 
convergence result (see, e.g., \cite[Theorem 23.41.(i)]{BauCom10}).

\begin{theorem}\label{ppa-maxmon}
Assume that $zer(A)\ne\emptyset$ and let $(\gamma_n)$ be a sequence of positive real numbers such that 
$\sum_{n=0}^\infty \gamma_n^2 = \infty$.  
For any $x \in H$, define the sequence $(x_n)$ by 
\begin{equation}\label{def-xn+1-JgammanAxn}
x_0:=x, \quad  x_{n+1}:=J_{\gamma_n A}x_n \, \text{~for all }n \in \mathbb{N}.
\end{equation}
Then $(x_n)$ converges weakly to a zero of $A$.
\end{theorem}
\details{
For all $n$, put $T_n:=J_{\gamma_n A}$. We have that $Fix(T_n)=zer(A)\ne\emptyset$ for every 
$n$ and, by Proposition~\ref{pj-maxmon}, the family $(T_n)_{n \in \mathbb{N}}$ is jointly firmly 
nonexpansive with respect to $(\gamma_n)_{n \in \mathbb{N}}$. Apply Theorem~\ref{appa-Hilbert} to derive our conclusion.
}

\section{Uniformly firmly nonexpansive and uniformly $(P_2)$ mappings}\label{sec:uniform}

As mentioned in the Introduction, if one wants to obtain strong convergence for the proximal point algorithm, one usually imposes a uniformity condition on the object that is being optimized. The aim of this section is to give such a condition in the abstract setting from the previous section. For a single mapping defined on a Hilbert space, this condition was also considered in \cite[Section 3.4]{BarBauMofWan16}, under the name of uniform firm nonexpansivity with a given modulus. We will now generalize this notion to CAT(0) spaces and show how it may be applied for the families of mappings that arise from two of the concrete problems just discussed.

Let $X$ be a CAT(0) space, $T: X \to X$, $C\subseteq X$  be a nonempty subset of $X$ 
and $\varphi : [0,\infty) \to [0,\infty)$ be an increasing function which vanishes only at $0$.

\begin{definition}
We say that $T$ is 
\be 
\item {\it uniformly firmly nonexpansive} on $C$ with modulus $\varphi$ if $T(C) \subseteq C$ 
and,  for all $x,y \in C$ and all $t \in [0,1]$, the following holds:
\begin{equation}\label{def-ufne}
d^2(Tx,Ty) \leq d^2((1-t)x+tTx,(1-t)y+tTy) - 2(1-t)\varphi(d(Tx,Ty)).
\end{equation}
\item {\it uniformly $(P_2)$} on $C$ with modulus $\varphi$ if $T(C) \subseteq C$ and, for 
any $x,y \in C$, 
\begin{equation}\label{def-up2}
2d^2(Tx,Ty) \leq d^2(x,Ty) + d^2(y,Tx) - d^2(x,Tx) - d^2(y,Ty) - 2\varphi(d(Tx,Ty)).
\end{equation}
\ee
\end{definition}
Obviously, if $T$ is uniformly firmly nonexpansive (resp. $(P_2)$) on $C$, then its restriction
$T|_C:C\to C$ is firmly nonexpansive (resp. $(P_2)$). We remark  also that the uniform $(P_2)$ 
condition may be expressed using the quasi-linearization function as follows:
\begin{equation}\label{uP2-BN}
\langle \vv{TxTy},\vv{yTy}\rangle \leq \langle \vv{TxTy} ,\vv{xTx}\rangle - \varphi(d(Tx,Ty)).
\end{equation}
\details{
\bua 
\langle \vv{TxTy},\vv{yTy}\rangle - \langle \vv{TxTy} ,\vv{xTx}\rangle&=& 
\frac12(d^2(Tx,Ty)+d^2(y,Ty)-d^2(Tx,y)-d^2(Ty,Ty))-\\
&& - \frac12(d^2(Tx,Tx)+d^2(Ty,x)-d^2(Tx,x)-d^2(Ty,Tx))\\
&=& \frac12(2d^2(Tx,Ty)+d^2(y,Ty)+d^2(x,Tx)-d^2(Tx,y)-d^2(Ty,x)).
\eua
}

\begin{proposition}
Suppose that $T$ is uniformly firmly nonexpansive on $C$ with modulus $\varphi$. 
Then $T$ is uniformly $(P_2)$ on $C$ with the same modulus $\varphi$.
\end{proposition}
\begin{proof}
Let $x,y \in C$ and $t\in(0,1)$. As in the proof of Proposition~\ref{jfn-jp2}, we apply the uniform firm nonexpansivity condition and \eqref{def-CAT0}  twice to get that
\begin{align*}
d^2(Tx,Ty)  \leq&\ (1-t)^2d^2(x,y) + t(1-t)d^2(Tx,y) + t(1-t)d^2(x,Ty) + t^2d^2(Tx,Ty)
\\ &\  - t(1-t)d^2(x,Tx) - t(1-t)d^2(y,Ty) - 2(1-t)\varphi(d(Tx,Ty)).
\end{align*}
Divide now by $1-t \neq 0$ to obtain that
\begin{align*}
(1+t)d^2(Tx,Ty)  \leq&\ (1-t)d^2(x,y) + td^2(Tx,y) + td^2(x,Ty) 
\\ &\ - td^2(x,Tx) - td^2(y,Ty) - 2\varphi(d(Tx,Ty)),
\end{align*}
By taking $t \to 1$ we get what is needed.
\end{proof}

As in the non-uniform case, for Hilbert spaces, the two notions coincide.

\begin{proposition}\label{H-uP2-ufne}
Assume that $X$ is a Hilbert space and that $T$ is uniformly $(P_2)$ on $C$ with modulus $\varphi$. 
Then $T$ is uniformly firmly nonexpansive on $C$ with the same modulus $\varphi$.
\end{proposition}
\begin{proof}
Let $x,y \in C$ and $t \in [0,1]$. By the hypothesis, \eqref{uP2-BN} and \eqref{eq-quasi-inner}, we immediately get that 
\[\langle Tx-Ty, (x-Tx) - (y-Ty) \rangle\geq  \varphi(\|Tx-Ty\|).\]

Consequently,
\begin{align*}
\|((1-t)x+tTx) - ((1-t)y+tTy)\|^2 =&\ \|(Tx-Ty) + (1-t)((x-Tx)-(y-Ty))\|^2\\
=&\ \|Tx-Ty\|^2 + (1-t)^2\|(x-Tx) - (y-Ty)\|^2 \\
&+ 2(1-t)\langle Tx-Ty, (x-Tx) - (y-Ty) \rangle \\
\geq&\ \|Tx-Ty\|^2 + 2(1-t)\varphi(\|Tx-Ty\|).
\end{align*}
\end{proof}

The following properties will be useful in the proof of our main quantitative result, Theorem \ref{conv-rate}.

\begin{lemma}\label{qp}
Let $T$ be uniformly $(P_2)$ on $C$ with modulus $\varphi$.  Then
\[
\varphi(d(Tx,z)) \leq d(x,Tx)d(Tx,z),
\]
for all $x \in C$ and all $z \in C \cap Fix(T)$.
\end{lemma}

\begin{proof}
Applying \eqref{def-up2} for $y:=z$, we get that
$$d^2(Tx,z) \leq d^2(x,z) - d^2(x,Tx) - 2\varphi(d(Tx,z)).$$
It follows that
\begin{align*}
2\varphi(d(Tx,z)) &\leq d^2(x,z) - d^2(x,Tx)- d^2(Tx,z)\\
&\leq (d(x,Tx) +d(Tx,z))^2 - d^2(x,Tx)- d^2(Tx,z)\\
&=2d(x,Tx)d(Tx,z).
\end{align*}
\end{proof}

As an immediate consequence, we obtain

\begin{corollary}\label{uP2-most-one-fp}
If $T$ is uniformly $(P_2)$ on $C$ with modulus $\varphi$, the set $C \cap Fix(T)$ is at most a singleton.
\end{corollary}
\begin{proof}
Let $x,z\in C \cap Fix(T)$. Applying  Lemma~\ref{qp}, we obtain that $\varphi(d(x,z)) = 0$. 
Since $\varphi$ vanishes only at $0$, we must have that $x=z$.
\end{proof}

We shall now check that the conditions introduced above are indeed satisfied by nontrivial 
particular cases in the concrete instances that we have presented.

\subsection{Uniformly convex functions}

Let $X$ be a complete CAT(0) space and $f: X \to (-\infty, \infty]$ be a proper, convex, lsc
function. We use the notation from Subsection \ref{minim-conv-lsc}.

Let $\psi : [0,\infty) \to [0,\infty)$ be an increasing function which vanishes only at $0$ 
and $C\subseteq X$ be nonempty. Recall that $f$ is said to be {\it uniformly convex} on $C$ 
with modulus $\psi$ if for all $x, y \in C$ 
and all $t \in [0,1]$, the following holds:
\begin{equation}\label{def-uc}
f((1-t)x+ty) \leq (1-t)f(x) + tf(y) - t(1-t)\psi(d(x,y)).
\end{equation}

\begin{lemma}
Assume that $f$ is uniformly convex on $C$ with modulus $\psi$. Let $\gamma >0$ be such that 
$J_{\gamma}(C) \subseteq C$. Then:
\be
\item for all $u,v\in C$, 
\begin{equation}\label{f-uc-lema-1}
d^2(J_{\gamma}u,v) \leq d^2(u,v) - d^2(u,J_{\gamma}u) - 2\gamma(f(J_{\gamma}u) - f(v)) 
- 2\gamma\psi(d(v,J_{\gamma}u)).
\end{equation}
\item for all $x,y \in C$, 
\begin{equation}\label{f-uc-lema-2}
d^2(J_{\gamma}x,J_{\gamma}y) \leq d^2(x,y) - 4\gamma\psi(d(J_{\gamma}x,J_{\gamma}y)).
\end{equation}
\ee
\end{lemma}
\begin{proof} 
\be
\item
By the definition of $J_{\gamma}$, we have that for all $p \in X$,
$$f(J_{\gamma}(u)) + \frac1{2\gamma}d^2(u,J_{\gamma}u) \leq f(p) + \frac1{2\gamma}d^2(u,p).$$
Let $t \in (0,1)$ be arbitrary. Note that, by \eqref{def-CAT0}, 
\begin{equation}\label{f-uc-lema-1-u1}
d^2(u,(1-t)v+tJ_{\gamma}u) \leq (1-t)d^2(u,v) + td^2(u,J_{\gamma}u) - t(1-t)d^2(v,J_{\gamma}u).
\end{equation}
Applying the first inequality (multiplied by $\gamma$) for $p: = (1-t)v + tJ_{\gamma}u$, \eqref{f-uc-lema-1-u1}
and the uniform convexity of $f$ on $C$ (since $v, J_{\gamma}u\in C$), we get that
\begin{align*}
\gamma f(J_{\gamma}u) + \frac12d^2(u,J_{\gamma}u) &\leq \gamma((1-t)f(v) + tf(J_{\gamma}u) - 
t(1-t)\psi(d(v,J_{\gamma}u))) \\
&\ \ \ + \frac12((1-t)d^2(u,v) + td^2(u,J_{\gamma}u) - t(1-t)d^2(v,J_{\gamma}u)),
\end{align*}
hence
\begin{align*}
\gamma(1-t)(f(J_\gamma u) - f(v)) \leq &\ \frac12 (1-t) (d^2(u,v) - d^2(u,J_\gamma u) - td^2(v,J_\gamma u) - 2\gamma t \psi(d(v,J_\gamma u))).
\end{align*}
\details{
\begin{align*}
\gamma f(J_{\gamma}u) + \frac12d^2(u,J_{\gamma}u) &\leq \gamma((1-t)f(v) + tf(J_{\gamma}u) - 
t(1-t)\psi(d(v,J_{\gamma}u))) \\
&\ \ \ + \frac12((1-t)d^2(u,v) + td^2(u,J_{\gamma}u) - t(1-t)d^2(v,J_{\gamma}u)),
\end{align*} 
\begin{align*}
\gamma f(J_{\gamma}u)  - \gamma(1-t)f(v) - \gamma tf(J_{\gamma}u)  &\leq - \frac12d^2(u,J_{\gamma}u) -\gamma
t(1-t)\psi(d(v,J_{\gamma}u))) \\
&\ \ \ + \frac12((1-t)d^2(u,v) + td^2(u,J_{\gamma}u) - t(1-t)d^2(v,J_{\gamma}u)),
\end{align*}
\begin{align*}
\gamma(1-t)(f(J_\gamma u) - f(v)) \leq &\ \frac12\bigg((1-t)d^2(u,v) + td^2(u,J_{\gamma}u) - 
t(1-t)d^2(v,J_{\gamma}u)- d^2(u,J_{\gamma}u)\\
&\ - 2\gamma
t(1-t)\psi(d(v,J_{\gamma}u)\bigg)
\end{align*} 
\begin{align*}
\gamma(1-t)(f(J_\gamma u) - f(v)) \leq &\ \frac12\left((1-t)d^2(u,v) - (1-t)td^2(u,J_{\gamma}u) - 
t(1-t)d^2(v,J_{\gamma}u) - 2\gamma
t(1-t)\psi(d(v,J_{\gamma}u)\right)
\end{align*}
\begin{align*}
\gamma(1-t)(f(J_\gamma u) - f(v)) \leq &\ \frac12(1-t)\left(d^2(u,v) - d^2(u,J_{\gamma}u) - 
td^2(v,J_{\gamma}u)- 2\gamma
t\psi(d(v,J_{\gamma}u)\right)
\end{align*}
}
Divide by $1-t \neq 0$ and let $t\to 1$ to  obtain that 
\begin{align*}
\gamma(f(J_\gamma u) - f(v)) \leq &\ \frac12  (d^2(u,v) - d^2(u,J_\gamma u) - d^2(v,J_\gamma u) - 
2\gamma  \psi(d(v,J_\gamma u))), 
\end{align*}
hence \eqref{f-uc-lema-1}.
\item Applying \eqref{f-uc-lema-1} with $u:=x, v:=J_{\gamma}y$ and then with $u:=y, v:=J_{\gamma}x$, we get that 
\begin{align*}
d^2(J_{\gamma}x,J_{\gamma}y) \leq &\ d^2(x,J_{\gamma}y) -d^2(x,J_{\gamma}x) -2\gamma(f(J_{\gamma}x) - 
f(J_{\gamma}y)) - 2\gamma\psi(d(J_{\gamma}x,J_{\gamma}y)),\\
d^2(J_{\gamma}y,J_{\gamma}x) \leq &\ d^2(y,J_{\gamma}x) -d^2(y,J_{\gamma}y) -2\gamma(f(J_{\gamma}y) 
- f(J_{\gamma}x)) - 2\gamma\psi(d(J_{\gamma}y,J_{\gamma}x)).
\end{align*}
Summing up, we obtain
$$2d^2(J_{\gamma}x,J_{\gamma}y) + d^2(x,J_{\gamma}x) + d^2(y,J_{\gamma}y) \leq d^2(x,J_{\gamma}y) + d^2(y,J_{\gamma}x) - 4\gamma\psi(d(J_{\gamma}y,J_{\gamma}x)).$$
By \eqref{bn}, we have that
$$d^2(x,J_{\gamma}y) + d^2(y,J_{\gamma}x) \leq d^2(x,y) + d^2(J_{\gamma}x,J_{\gamma}y) + d^2(x,J_{\gamma}x) + d^2(y,J_{\gamma}y),$$
from where we get our conclusion.
\ee
\end{proof}

\begin{proposition}\label{f-uc-Jgamma-ufne}
Suppose that $f$ is uniformly convex on $C$ with modulus $\psi$. Let $\gamma >0$ be such that 
$J_{\gamma}(C) \subseteq C$.
Then $J_{\gamma}$ is uniformly firmly nonexpansive on $C$ with modulus $2\gamma\psi$.
\end{proposition}
\begin{proof}
Let $x,y \in C$ and $t \in [0,1]$. Denote $u:=(1-t)x+tJ_{\gamma}x$ and 
$v:=(1-t)y + tJ_{\gamma}y$. 

 By Proposition~\ref{Jgamaf-prop}.\eqref{Jgamaf-jos}, we have that $J_{(1-t)\gamma}(u)=J_\gamma x$ and 
$J_{(1-t)\gamma}(v)=J_\gamma y$.  We get that 
\begin{align*}
d^2(J_{\gamma}x,J_{\gamma}y) =&\ d^2(J_{(1-t)\gamma}(u), J_{(1-t)\gamma}(v)) \\
\leq&\  d^2(u,v) - 4(1-t)\gamma\psi(d(J_{(1-t)\gamma}(u), J_{(1-t)\gamma}(v)) \quad\text{by \eqref{f-uc-lema-2}}\\
=&\  d^2(u,v) - 4(1-t)\gamma\psi(d(J_\gamma x, J_\gamma y).
\end{align*}
\end{proof}

\subsection{Uniformly monotone operators}

Fix now a Hilbert space $H$ and $C\subseteq H$ a nonempty subset. Let $A : H \to 2^H$ be a multi-valued operator and 
$\varphi : [0,\infty) \to [0,\infty)$ be an increasing function which vanishes only at $0$.

Then $A$ is said to be {\it uniformly monotone} on $C$ with modulus $\varphi$ (see, e.g. \cite[Definition 22.1]{BauCom10}) if for all $x,y \in C$ 
and $u,v \in H$ with $u \in A(x)$ and $v \in A(y)$ we have that
$$\langle x-y,u-v \rangle \geq \varphi(\|x-y\|).$$

\begin{proposition}\label{unif-mon-ufne}
Assume that $A$ is a maximally monotone operator which is uniformly monotone 
on $C$ with modulus $\varphi$. Let $\gamma >0$ be such that 
$J_{\gamma A}(C) \subseteq C$. Then $J_{\gamma A}$ is uniformly firmly 
nonexpansive on $C$ with modulus $\gamma\varphi$.
\end{proposition}
\begin{proof}
Let $x,y \in C$. As in the proof of Proposition \ref{pj-maxmon}, we get that 
$$\langle J_{\gamma A}x - J_{\gamma A}y, x- J_{\gamma A}x \rangle \geq \langle J_{\gamma A}x-J_{\gamma A}y , y - J_{\gamma A}y \rangle + 
\gamma\varphi(\|J_{\gamma A}x-J_{\gamma A}y\|).$$
Thus, $J_{\gamma A}$ is uniformly $(P_2)$ on $C$ with modulus $\gamma\varphi$. Apply now Proposition~\ref{H-uP2-ufne}.
\end{proof}

\section{A rate of convergence for the uniform case}\label{sec:quantitative}

We shall now show that in the presence of the uniformity constraint described in the previous section, 
one indeed gets strong convergence of the proximal point algorithm in its most abstract form, 
given by Theorem~\ref{gppa}. 
Moreover, as announced in the Introduction, we use the tools of proof mining to derive that 
result from a stronger one which is highly uniform and also quantitative -- i.e. also yields 
a {\it rate of convergence} for the sequence. 

Let us recall that if $(a_n)_{n \in \mathbb{N}}$ is a convergent sequence in a metric space $(X,d)$ with 
$\lim_{n\to\infty}a_n=a$, then a {\it rate of convergence} of $(a_n)$ is a function $\Phi : \mathbb{N} \to \mathbb{N}$ 
such that for all $k \in \mathbb{N}$ and all $n \geq \Phi(k)$,
$$d(a_n,a) \leq \frac1{k+1}.$$
Another needed quantitative notion will be that of a {\it rate of divergence} for a given diverging series 
$\sum_{n=0}^\infty b_n = \infty$,  which is a function $\theta : \mathbb{N} \to \mathbb{N}$ such that for all 
$K \in \mathbb{N}$ we have that
$\sum_{n=0}^{\theta(K)} b_n \geq K$. 

\mbox{}

In this section, $X$ is a complete CAT(0) space and $T_n: X \to X$ for every $n\in \N$. We assume that 
the family $(T_n)$ has common fixed points and set $$F:=\bigcap_{n \in \mathbb{N}} Fix(T_n)\neq\emptyset.$$
Furthermore, $\varphi:[0,\infty) \to [0,\infty)$ is an increasing function which vanishes only at $0$ and 
$(\gamma_n)$ is a a sequence in $(0,\infty)$ such that $\sum_{n=0}^\infty \gamma_n^2 = \infty$ 
with rate of divergence $\theta$.

We can state now the main result of this section.  

\begin{theorem}\label{conv-rate}
Let $b\in\N, p\in F$ and $C$ be the closed ball of center $p$ and radius $b$. 
Assume that, for all $n\in\N$, $T_n$ is uniformly $(P_2)$ on $C$ with modulus $\gamma_n\varphi$.

For every $x\in C$, let $(x_n)$ be defined by
\beq
x_0:=x, \quad x_{n+1}:=T_nx_n \text{~for all~}n\in\mathbb{N}. \label{def-ppa-final}
\eeq
Suppose that  $(C2)$ holds, that is, the sequence $\left(\frac{d(x_n,x_{n+1})}{\gamma_n}\right)$ is nonincreasing. 

Then $C\cap F=\{p\}$ and $(x_n)$ converges strongly to $p$ with rate of convergence 
$\Psi_{b,\theta,\varphi}$,  given by
\begin{equation}\label{def-Psi-b-theta-vp}
\Psi_{b,\theta,\varphi}(k):=\Sigma_{b,\theta}\left(\left\lceil \frac{2b}{\varphi\left(\frac1{k+1}\right)} 
\right\rceil \right) + 1,
\end{equation}
with $\Sigma_{b,\theta}(k):=\theta(b^2(k+1)^2)$.
\end{theorem}

Before proving the theorem, let us give some consequences.

\begin{proposition}\label{ppa-maxmon-quant}
Assume that $H$ is a Hilbert space and $A:H \to 2^H$ is a maximally monotone operator with $zer(A)\ne\emptyset$.
Let $b\in\N, p\in zer(A)$ and $C$ be the closed ball of center $p$ and radius $b$.
Suppose that $A$ is uniformly monotone  on $C$ with modulus $\varphi$.  
For any $x \in C$, let $(x_n)$ be defined by \eqref{def-xn+1-JgammanAxn}. 

Then $p$ is the unique zero of $A$ in $C$ and $(x_n)$ converges strongly  to $p$ with rate of convergence 
$\Psi_{b,\theta,\varphi}$, given by \eqref{def-Psi-b-theta-vp}.
\end{proposition}
\begin{proof}  
We use the notation from Subsection~\ref{concrete-zeros-maximal}. Since, for every $n\in\N$, 
$Fix(J_{\gamma_nA})=zer(A)\ne\emptyset$ and $J_{\gamma_nA}$ is nonexpansive, it is obvious that $J_{\gamma_nA}(C)\subseteq C$. Thus,
by Proposition~\ref{unif-mon-ufne}, every $J_{\gamma_nA}$ is uniformly firmly 
nonexpansive on $C$ with modulus $\gamma_n\varphi$. Furthermore, $(C2)$ is satisfied, by Propositions~\ref{pj-maxmon} 
and \ref{cs}. An application of Theorem~\ref{conv-rate} for the family $(J_{\gamma_nA})$ yields the 
result.
\end{proof}

The above proposition is a quantitative uniform version of Theorem~\ref{ppa-maxmon}. If we forget about the quantitative 
features, we get immediately the following well-known result (see, e.g., \cite[Theorem 23.41.(ii)]{BauCom10}).

\begin{corollary}
Assume that $H$ is a Hilbert space and $A:H \to 2^H$ is a maximally monotone operator with $zer(A)\ne\emptyset$.
Let $(\gamma_n)$ be a sequence of positive real numbers such that $\sum_{n=0}^\infty \gamma_n^2 = \infty$, $x\in X$ 
and $(x_n)$ be defined by \eqref{def-xn+1-JgammanAxn}.
Suppose that $A$ is uniformly monotone on every bounded subset of $H$.

Then $(x_n)$ converges strongly to the unique zero of $A$.
\end{corollary} 
\details{Let $p\in zer(A)$, $b\in \N$ be such that $d(x,p)\leq b$ and $C$ be the closed ball of center $p$ and radius $b$.
Apply the fact that $A$ is uniformly monotone on $C$ to get that $(x_n)$ converges strongly to $p$. 
It is well-known that $A$ has a unique zero.
}

The following result is a quantitative uniform version of Theorem~\ref{ppa-lsc}.

\begin{proposition}
Assume that $X$ is a complete CAT(0) space and $f: X \to (-\infty, \infty]$ is a convex, proper,
lsc function that attains its minimum.
Let $b\in\N, p\in Argmin(f)$ and $C$ be the closed ball of center $p$ and radius $b$.
Suppose that $f$ is  uniformly convex on $C$ with modulus $\psi$. For any $x \in C$, let $(x_n)$ 
be defined by \eqref{def-lsc-ppa}. 

Then $p$ is the unique minimizer of $f$ in $C$ and $(x_n)$ converges strongly to $p$ with rate of convergence 
$\Omega_{b,\theta,\psi}:=\Psi_{b,\theta,2\psi}$. 
\end{proposition}
\begin{proof}  
By Proposition~\ref{Jgamaf-prop}, $Fix(J_{\gamma_n})=Argmin(f)\ne\emptyset$ and $J_{\gamma_n}$ is nonexpansive, hence 
$J_{\gamma_n}(C)\subseteq C$ for all $n$. Use now Proposition~\ref{f-uc-Jgamma-ufne} to get that 
every $J_{\gamma_n}$ is uniformly firmly nonexpansive on $C$ with modulus $2\gamma_n\psi$. Since $(C2)$ is satisfied
(by Propositions~\ref{Jgammaf-jfne} and \ref{cs}), we can apply  Theorem~\ref{conv-rate} for the family $(J_{\gamma_n})$ 
to get the result. 
\end{proof}

\subsection{Proof of Theorem~\ref{conv-rate}}

Apply the fact that $C\cap F\ne\emptyset$ and Corollary~\ref{uP2-most-one-fp} to conclude 
that  $C\cap F=\{p\}$. \details{By hypothesis, $p\in C\cap F$. Let $n\in \N$ be arbitrary. Then  
we have that $p\in C\cap Fix(T_n)$. Since $T_n$ is uniformly $P_2$ on $C$, it follows, by Corollary~\ref{uP2-most-one-fp}, that 
$C\cap Fix(T_n)=\{p\}$. This hold for all $n\in\N$, so $C\cap F=\{p\}$.} 
Since, by Lemma~\ref{2-lem-1}, $(x_n)$ is Fej\'er monotone 
with respect to $F$, we have that $d(x_n,p)\leq b$ for all $n\in\N$. 

\noindent {\bf Claim:} $\Sigma_{b,\theta}$ is a rate of convergence of the sequence 
$\left(\frac{d(x_n,x_{n+1})}{\gamma_n}\right)$
towards $0$.\\[1mm]
{\bf Proof of claim:}  We reason as in the proof of \cite[Lemma 8.3.(ii)]{KohLeuNic18}. 
Let $k \in \mathbb{N}$. By the proof of Lemma~\ref{2-lem-2}, 
\begin{equation}\label{claim-main-1}
\sum_{n=0}^{\infty}d^2(x_n,x_{n+1})\leq b^2. 
\end{equation}
Assume that for all  $n\in\{0,\ldots, \Sigma_{b,\theta}(k)\}$ we have that 
$\frac{d(x_n,x_{n+1})}{\gamma_n}>\frac1{k+1}$. It follows that 
$$\sum_{n=0}^{\Sigma_{b,\theta}(k)}d^2(x_n,x_{n+1}) > 
\sum_{n=0}^{\Sigma_{b,\theta}(k)}\gamma_n^2\frac1{(k+1)^2}=
 \frac1{(k+1)^2}\sum_{k=0}^{\theta(b^2(k+1)^2)}\gamma_n^2\geq b^2.$$
We get a contradiction with \eqref{claim-main-1}. Thus, there exists $N\leq \Sigma_{b,\theta}(k)$ such that 
$\frac{d(x_N,x_{N+1})}{\gamma_n}\leq \frac1{k+1}$. By $(C2)$, the claim follows. 
\hfill $\blacksquare$\\[2mm]

Let $k \in \mathbb{N}$ and $n \geq \Psi_{b,\theta,\varphi}(k)$. Set $n':=n-1$. Then
$n' \geq \Sigma_{b,\theta}\left(\left\lceil \frac{2b}{\varphi\left(\frac1{k+1}\right)} \right\rceil \right),$
hence, by the claim,
$$\frac{d(x_{n'},x_{n'+1})}{\gamma_{n'}} \leq 
\frac1{\left\lceil \frac{2b}{\varphi\left(\frac1{k+1}\right)} \right\rceil + 1}
\leq \frac1{\frac{2b}{\varphi\left(\frac1{k+1}\right)}} = \frac1{2b}\cdot\varphi\left(\frac1{k+1}\right).$$

Applying Lemma~\ref{qp} for $x:=x_{n'}$, $z:=p$, $T:=T_{n'}$ (and hence $\varphi$ becomes $\gamma_{n'}\varphi$), we get that
$$\gamma_{n'}\varphi(d(T_{n'}x_{n'},p)) \leq d(x_{n'},T_{n'}x_{n'})d(T_{n'}x_{n'},p).$$
Since $x_{n'+1} = T_{n'}x_{n'}$, it follows that
$$\varphi(d(x_{n'+1},p)) \leq \frac{d(x_{n'},x_{n'+1})}{\gamma_{n'}} \cdot d(x_{n'+1},p) 
\leq \frac1{2b} \cdot \varphi\left(\frac1{k+1}\right) \cdot b = \frac12 \varphi\left(\frac1{k+1}\right).$$
If $d(x_{n'+1},p) > \frac1{k+1}$, then $\varphi(d(x_{n'+1},p))\geq \varphi\left(\frac1{k+1}\right)
>\frac12 \varphi\left(\frac1{k+1}\right)$, since $\varphi$ is increasing and $\varphi\left(\frac1{k+1}\right)\ne 0$. 
We have got a contradiction. Thus, we must have
$$d(x_{n'+1},p)\leq \frac1{k+1},$$
which is what we wanted to show, since $n = n' +1$.\hfill \qed

\mbox{}

\noindent
{\bf Acknowledgements:} \\[1mm] 
Adriana Nicolae was partially supported by DGES (MTM2015-65242-C2-1-P). She would also like to acknowledge the Juan de la Cierva - Incorporaci\'{o}n Fellowship Program of the Spanish Ministry of Economy and Competitiveness.\\
Lauren\c tiu Leu\c stean and Andrei Sipo\c s were partially supported by a grant of the Romanian 
National Authority for Scientific Research, CNCS - UEFISCDI, project 
number PN-II-ID-PCE-2011-3-0383.

\end{document}